\documentclass[12pt]{amsart}
\usepackage{amscd,amsmath,amssymb,amsfonts}
\usepackage[cmtip, all]{xy}

\newtheorem{thm}{Theorem}[section]
\newtheorem{prop}[thm]{Proposition}
\newtheorem{lem}[thm]{Lemma}
\newtheorem{cor}[thm]{Corollary}
\newtheorem{conj}[thm]{Conjecture}

\theoremstyle{definition}
\newtheorem{exm}[thm]{Example}
\newtheorem{defn}[thm]{Definition}

\theoremstyle{remark}
\newtheorem{remk}[thm]{Remark}
\newtheorem{remks}[thm]{Remarks}
\newtheorem{exms}[thm]{Examples}
\newtheorem{notat}[thm]{Notation}

\numberwithin{equation}{section}

{\hfill$\square$\end{defn}}

{\hfill$\square$\end{remk}}

{\hfill$\square$\end{remks}}

{\hfill$\square$\end{exm}}

{\hfill$\square$\end{exms}}

{\hfill$\square$\end{notat}}

\newcommand{\sC}{{\mathcal C}}

\newcommand{\sE}{{\mathcal E}}
\newcommand{\sF}{{\mathcal F}}

\newcommand{\sH}{{\mathcal H}}
\newcommand{\sI}{{\mathcal I}}

\newcommand{\sK}{{\mathcal K}}

\newcommand{\sN}{{\mathcal N}}
\newcommand{\sO}{{\mathcal O}}
\newcommand{\sP}{{\mathcal P}}

\newcommand{\wt}{\widetilde}

\newcommand{\A}{{\mathbb A}}

\newcommand{\F}{{\mathbb F}}

\renewcommand{\H}{{\mathbb H}}

\newcommand{\Q}{{\mathbb Q}}

\newcommand{\Z}{{\mathbb Z}}
\renewcommand{\1}{{\mathbb{S}}}

\newcommand{\surj}{\twoheadrightarrow}

\newcommand{\ds}{{/\kern-3pt/}}

\begin{document}
\title{On the negative $K$-theory of schemes in finite characteristic}
\author{Amalendu Krishna}


\baselineskip=10pt 
                                              
\begin{abstract}
We show that if $X$ is a $d$-dimensional scheme of finite type over a 
perfect field $k$ of characteristic $p > 0$, then $K_i(X) = 0$
and $X$ is $K_i$-regular for $i < -d-2$ whenever the resolution of singularities
holds over $k$. This proves the $K$-dimension conjecture of Weibel
\cite[2.9]{Weibel1} (except for $-d-1 \le i \le -d-2$) in all characteristics, 
assuming the resolution of singularities.  
\end{abstract} 

\maketitle

\section{Introduction}
It is by now well known that the negative $K$-theory of singular schemes
is non-zero in general and bears a significant information about the nature of 
the singularity of $X$. Hence it is a very interesting question to know
how much of the negative $K$-theory of a singular variety can survive.
A beautiful answer was given in terms of the following very general 
conjecture of Weibel.
\begin{conj}[Weibel, \cite{Weibel1}]
Let $X$ be a Noetherian scheme of dimension $d$. Then $K_i(X) = 0$ for 
$i < -d$ and $X$ is $K_{-d}$-regular.
\end{conj}
This conjecture was proved recently by Cortinas-Haesemeyer-Schlichting-Weibel
\cite[Theorem~6.2]{CHSW} for schemes of finite type over a field of
characteristic zero. If $X$ is a scheme of finite type over a field of
positive characteristic, the above conjecture was proved by Weibel
\cite{Weibel2} provided the dimension of $X$ is at most two. Our aim in this
paper is to prove the conjecture for a $d$-dimensional scheme in the positive
characteristic (except for $i = -d-1$) assuming the resolution of 
singularities. A {\sl variety} $X$ in this paper will mean a scheme of
finite type over an infinite perfect field $k$ of characteristic $p > 0$, 
which will be fixed throughout this paper.

\begin{defn}[Resolution of singularities] We will say that the resolution
of singularities holds over $k$ if given any equidimensional scheme $X$
of finite type over $k$, there exists a sequence of monoidal transformations
\[
X_r \to X_{r-1} \to \cdots \to X_1 \to X_0 = X
\]
such that the following hold: \\
$(i)$ the reduced subscheme $X^{\rm red}_r$ is smooth over $k$, \\
$(ii)$ the center $D_i$ of the monoidal transformation $X_{i+1} \to X_i$
is smooth and connected and nowhere dense in $X_i$.
\end{defn}
The resolution of singularities holds over fields of characteristic zero
by the work of Hironaka \cite[Theorem~1*]{Hironaka1}. For fields of positive
characteristics, this problem is still not known though widely expected to
be true. Recently, Hironaka \cite[]{Hironaka2} has outlined a complete program
to solve this problem and work on this program is in progress.

Recall that a variety $X$ is said to be $K_i$-regular if the natural
map $K_i(X) \to K_i(X[T_1, \cdots , T_r])$ is an isomorphism for all 
$r \ge 1$, where $K_i(X)$ is the $i$th stable homotopy group of the 
non-connective spectrum $K(X)$ of perfect complexes on $X$. 
It is known from a result of Vorst \cite{Vorst} that a scheme which is
$K_i$-regular, is also $K_j$-regular for $j \le i$.
We now state the main result of this paper.
\begin{thm}\label{thm:main}
Assume that the resolution of singularities holds over $k$ and let $X$ be a variety 
of dimension $d$ over $k$. Then \\
$(i) \ \ K_i\left(X, {\Z}/n\right) = 0$ for $i < -d-1$ and for all
$n \ge 1$. \\
$(ii) \ \ K_{-d-2}\left(X[T_1, \cdots , T_j]\right)$ is a divisible
torsion group for all $j \ge 0$. \\ 
$(iii)\ \ K_i(X) = 0$ and $X$ is $K_i$-regular for $i < -d-2$. 
\end{thm}    
\section{$K$-theory and topological cyclic homology}
In this section, we briefly recall the topological cyclic homology
of rings and schemes and the cyclotomic trace map from the $K$-theory 
to the topological cyclic homology. We then show that the homotopy
fiber of this trace map satisfies the descent for the $cdh$-topology.

Let $p$ be a fixed prime number. Let $A$ be a commutative ring which is 
essentially of finite type over a field $k$ of characteristic $p$.
Recall from \cite{GH1} that the topological Hochschild spectrum $T(A)$
is a symmetric ${\1}$-spectrum, where $\1$ is the circle group. Let
$C_{r} \subset {\1}$ be the cyclic subgroup of order $r$. Then one
defines 
\[
TR^{n}\left(A;p\right) = {F\left({\1}/{C^{p^{n-1}}}, T(A)\right)}^{\1}
\]
to be the fixed point spectrum of the function spectrum 
$F\left({\1}/{C^{p^{n-1}}}, T(A)\right)$. There are the frobenius and the
restriction maps of spectra 
\[
F, R : TR^{n}\left(A;p\right) \to TR^{n-1}\left(A;p\right).
\]
The spectrum $TC^{n}$ is defined as the
homotopy equalizer of the maps $F$ and $R$, i. e.,
\[
TC^{n}\left(A;p\right) = eq\left(TR^{n}\left(A;p\right)
\stackrel{F} {\underset {R} \rightarrow} TR^{n}\left(A;p\right)\right),
\]
and the topological cyclic homology spectrum $TC\left(A;p\right)$ is 
defined as the homotopy limit
\[
TC\left(A;p\right) = {\rm holim} {\rm TC}^{n}\left(A;p\right).
\] 
One similarly defines 
\[
TR\left(A;p\right) = {\underset {R} {\rm holim}} {\rm TR}^{n}\left(A;p\right)
\]
\[
TF\left(A;p\right) = {\underset {F} {\rm holim}} {\rm TR}^{n}\left(A;p\right).
\] 
It was shown by Geisser and Hesselholt in {\sl loc. cit.} that the
topological Hochschild and cyclic homology satisfy descent for
a Cartesian diagram of rings and they were able to define these
homology spectra for a Noetherian scheme $X$ using the Thomason's 
construction of the hypercohomology spectrum \cite[1.33]{Thomason1}.
Recently, Blumberg and Mandell \cite{BM} have made a significant progress in 
the study of the topological Hochschild and cyclic homology of schemes.
They globally define the spectra $T(X)$ and $TC(X)$ for a Noetherian scheme 
$X$ as the topological Hochschild and cyclic homology spectra of the
spectral category $D\left(Perf/X\right)$ which is the Thomason's
derived category of perfect complexes on $X$ \cite{Thomason2}.
They show that their definition of these spectra coincides with the
above definition for affine schemes. They also prove the localization
and the Zariski descent properties of the topological Hochschild and cyclic
homology of schemes. We refer to \cite{BM} for more details. In this paper,
the topological Hoschschild and cyclic homology of schemes will be 
considered in the sense of \cite{BM}. For any symmetric spectrum $E$ and
for $n \ge 1$, let $E/{p^n}$ denote the smash product of $E$ with a
mod $p^n$ Moore spectrum ${\Sigma}^{\infty}/{p^n}$.

Let $K(X)$ denote the Thomason's non-connective spectrum of the 
perfect complexes on $X$. For a ring $A$, there is a cyclotomic trace map
\cite{Boks} of non-connective spectra
\[
K/{p^n} (A) \xrightarrow{tr} TC/{p^n}\left(A;p\right).
\]
Since $K$-theory satisfies Zariski descent by \cite{Thomason2} and so
does the topological cyclic homology by \cite{BM}, taking the induced
map on the Zariski hypercohomology spectra gives for any Noetherian scheme
$X$, the cyclotomic trace map of spectra
\begin{equation}\label{eqn:trace}
K/{p^n} (X) \xrightarrow{tr} TC/{p^n}\left(X;p\right).
\end{equation}
Let $L^n(X)$ denote the homotopy fiber of the trace map in ~\ref{eqn:trace}.
If ${Sch}/{k}$ denotes the category of varieties over $k$, then one gets
a presheaf of homotopy fibrations of spectra on ${Sch}/{k}$
\begin{equation}\label{eqn:trace1}
L^n \to K/{p^n} \xrightarrow{tr} TC/{p^n}\left(-;p\right).
\end{equation}

\section{cdh-Descent for $L^n$}
We recall from \cite{CHSW} that a presheaf of spectra $\sE$ on the 
category $Sch/k$ satisfies the {\sl Mayer-Vietoris property} for a
Cartesian square of schemes
\begin{equation}\label{eqn:Car}
\xymatrix{
Y' \ar[r] \ar[d] & X' \ar[d] \\
Y \ar[r] & X}
\end{equation}
if applying $\sE$ to this square results in a homotopy Cartesian
square of spectra. We say that $\sE$ satisfies the Mayer-Vietoris
property for a class of squares provided it satisfies this property
for each square in that class. One says that the presheaf of spectra
$\sE$ is {\sl invariant under infinitesimal extension} if for any affine
scheme $X$ and a closed subscheme $Y$ of $X$ defined by a sheaf of nilpotent ideals
$\sI$, the spectrum ${\sE}\left(X, Y\right)$ is contractible, where
the latter is the homotopy fiber of the map ${\sE}(X) \to {\sE}(Y)$.
One says that $\sE$ satisfies the
{\sl excision property} if for any morphism of affine schemes 
$f : X \to Y$ and a sheaf of ideals $\sI$ on $X$ such that 
$I \cong f_* f^*(\sI)$, the spectrum $\sE\left(Y, X, \sI\right)$ is
contractible, where $\sE\left(X, Y, \sI\right)$ is defined as the
homotopy fiber of the map ${\sE}\left(X, \sI\right)
\to {\sE}\left(Y, \sI\right)$. 
An elementary Nisnevich square is a 
Cartesian square of schemes as above such that $Y \to X$ is an open
embedding, $X' \to X$ is {\'e}tale and $(X' - Y') \to (X - Y)$ is an
isomorphism. Then one says that $\sE$ satisfies {\sl Nisnevich descent}
if it satisfies the Mayer-Vietoris property for all elementary 
Nisnevich squares.   

We next recall from \cite{Voev} (see also {\sl loc. cit.}) that a 
$cd$-structure on a small
category $\sC$ is a class $\sP$ of commutative squares in $\sC$ that is
closed under isomorphisms. Any such $cd$-structure defines a topology
on $\sC$. We assume in the rest of the paper that our ground field $k$
admits the resolution of singularities. In this case, the 
{\sl combined cd-structure} on the category $Sch/k$ consists of all
elementary Nisnevich squares and all abstract blow ups, where an
abstract blow-up is a Cartesian square as in ~\ref{eqn:trace1} where
$Y \to X$ is a closed embedding, $X' \to X$ is proper and the induced
map $(X'- Y')_{red} \to (X - Y)_{red}$ is an isomorphism.
The topology generated by the combined $cd$-structure is called the
$cdh$-topology. Let $Sm/k$ denote the category of smooth varieties over 
$k$. 
Since the resolution of singularities holds over $k$, the restriction of   
the $cd$-structure to the category $Sm/k$ where abstract blow-ups are
replaced by the smooth blow-ups, is complete, bounded and regular
({\sl cf.} \cite[Section~4]{Voev}). The topology generated by this
$cd$-structure on $Sm/k$ is called the $scdh$-topology. This is just the
restriction of the $cdh$-topology on the subcategory $Sm/k$.
In this paper, we shall consider the local injective model structure
on the category of presheaves of spectra on $Sch/k$ as described in 
\cite{CHSW}. 

For the local injective model structure the category of presheaves of
spectra on a given topology $\sC$ on $Sch/k$, a fibrant replacement
of a presheaf of spectra $\sE$ is a trivial cofibration $\sE \to {\sE'}$
where ${\sE'}$ is fibrant. We shall write such a fibrant replacement
as $\H_{\sC}(-, \sE)$. We shall say that $\sE$ satisfies the $cdh$-descent
if it satisfies the Mayer-Vietoris property for all elementary
Nisnevich squares and all abstract blow-ups. By 
\cite[Theorem~3.4]{CHSW}, this is equivalent to the assertion that the
map $\sE \to \H_{cdh}(-, \sE)$ is a global weak equivalence in the sense
that $\sE(X) \to \H_{cdh}(X, \sE)$ is a weak equivalence for all $X \in Sch/k$.   
\begin{thm}\label{thm:descent}
Let $\sE$ be a presheaf of spectra on $Sch/k$ such that $\sE$ satisfies 
excision, is invariant under infinitesimal extension, satisfies Nisnevich
descent and satisfies the Mayer-Vietoris property for every blow-up along
a regular closed embedding. Then $\sE$ satisfies the $cdh$-descent.
\end{thm}

\begin{proof} The proof of this theorem is very similar to the proof
of the analogous theorem in {\sl loc. cit.} (Theorem~3.12) when $k$ has
characteristic zero. We only give the brief sketch. As shown above, it
suffices to show that the map 
\begin{equation}\label{eqn:descent1}
\sE(X) \to \H_{cdh}(X, \sE)
\end{equation}
is a weak equivalence for all varieties $X$ over $k$. Since the 
$scdh$-topology on $Sm/k$ is generated by elementary Nisnevich squares and 
smooth blow-ups and since the closed embeddings of smooth varieties are
regular embeddings, we see that $\sE$ satisfies the $scdh$-descent
in $Sm/k$.

Now assume $X$ is singular. As explained in {\sl loc. cit.}, the 
argument goes as in the proof of Theorem~6.4 in \cite{Haes}.
The excision, invariance under infinitesimal extension and Nisnevich
descent together imply that $\sE$ satisfies the Mayer-Vietoris property
for closed covers and for finite abstract blow-ups. Now if $X$ is a 
hypersurface in a smooth scheme, we can follow the proof of
Theorem~6.1 in \cite{Haes} to conclude that ~\ref{eqn:descent1} holds for
$X$ since the resolution of singularities holds over $k$, which is
also infinite. If $X$ is a complete intersection inside a smooth $k$-scheme,
then we can use the hypersurface case, the Mayer-Vietoris for the closed
covers and an induction on the embedding dimension of $X$ to conclude
~\ref{eqn:descent1} for $X$. The general case follows from this as
shown in \cite[Theorem~6.4]{Haes}.
\end{proof} 
\begin{cor}\label{cor:descentL}
The presheaf of spectra $L^n$ ({\sl cf.} ~\ref{eqn:trace1}) satisfies
the $cdh$-descent.
\end{cor}
\begin{proof} We need to show that $L^n$ satisfies all the conditions
of Theorem~\ref{thm:descent}. We have the homotopy fibration of
presheaves of spectra
\[
L^n \to K/{p^n} \to TC/{p^n}\left(-;p\right).
\]
The fact that $L^n$ satisfies excision was proved by Geisser-Hesselholt
\cite[Theorem~1]{GH2}. The invariance of $L^n$ under infinitesimal
extension was proved by McCarthy \cite[Main Theorem]{Mac}.
Next we show that $L^n$ satisfies Nisnevich descent.
$K/{p^n}$ satisfies Nisnevich descent by \cite[Theorem~10.8]{Thomason2}.
$TC/{p^n}\left(-;p\right)$ satisfies Nisnevich descent by
\cite[Corollary~3.3.4]{GH1} and by the agreement of the definition
of the topological cyclic homology as given in \cite{GH1} with that
of \cite{BM} since the topological cyclic homology of \cite{BM}
satisfies the Zariski descent (see the discussion in Remark~3.3.5
of \cite{GH1}). We now consider the following commutative diagram
of spectra for a given variety $X$.
\begin{equation}\label{eqn:spectra*}
\xymatrix{
L^n(X) \ar[r] \ar[d] & K/{p^n}(X) \ar[r] \ar[d] & TC/{p^n}\left(X;p\right) 
\ar[d] \\
\H_{Nis}\left(X, L^n\right) \ar[r] & \H_{Nis}\left(X, K/{p^n}\right) \ar[r]
& \H_{Nis}\left(X, TC/{p^n}\left(-;p\right)\right)}
\end{equation}
Since the top row in the above diagram is a homotopy fibration and 
the bottom row is a fibrant replacement of the top row, the bottom row
is also a homotopy fibration ({\sl cf.} \cite[1.35]{Thomason1},
see also \cite[Section~5]{CHSW}). Now, since the middle and the right
vertical maps are weak equivalences, we see that that the left vertical
map is also a weak equivalence. This verifies the Nisnevich descent for
$L^n$. Finally, $L^n$ satisfies the Mayer-Vietoris property for the
blow-up along regular closed embeddings by \cite[Theorem~1.4]{BM}. 
We conclude from Theorem~\ref{thm:descent} that $L^n$ satisfies 
$cdh$-descent.
\end{proof}
Let $a$ denote the natural morphism from the $cdh$-site to the Zariski
site on the category $Sch/k$. For any Zariski sheaf $\sF$, let
$a_{cdh} {\sF}$ denote the $cdh$-sheafification  of $\sF$. 
\begin{cor}\label{cor:descentL*}
For any $k$-variety $X$, there is a strongly convergent spectral
sequence 
\[
E^{p,q}_2 = H^p_{cdh}\left(X, a_{cdh}{\pi}_q(L^n)\right) 
\Rightarrow L^n_{q-p}(X),
\]
where the differentials of the spectral sequence are $d_r :
E^{p,q}_r \to E^{p+r, q+r-1}_r$.
\end{cor}
\begin{proof} This follows immediately from Corollary~\ref{cor:descentL}
and the fact that the $cdh$-cohomological dimension of $X$ is bounded by
the Krull dimension of $X$ ({\sl cf.} \cite[Theorem~12.5]{SuslinV}.
\end{proof}

\section{Vanishing and homotopy invariance for $L^n$}  
Following \cite{CHSW}, we let ${\wt{C}}_j{\sE}$ denote the homotopy
cofiber of the natural map $\sE \to {\sE}\left(- \times {\A}^j\right)$
for any presheaf of spectra $\sE$ on $Sch/k$. Note that 
${\sE}\left(- \times {\A}^j\right)$ is a canonical direct sum of
$\sE$ and ${\wt{C}}_j{\sE}$ and hence the functor ${\wt{C}}_j$
preserves the homotopy fibration sequences. In particular, we
get a presheaf of fibration sequences
\begin{equation}\label{eqn:Cfib}
{\wt{C}}_j{L^n} \to {\wt{C}}_j{K/{p^n}} \to {\wt{C}}_j
{TC/{p^n}\left(-;p\right)}.
\end{equation} 
Furthermore, since $L^n$ and $L^n\left(- \times {\A}^j\right)$ satisfy
$cdh$-descent, we see that ${\wt{C}}_j{L^n}$ also satisfies  
$cdh$-descent. 
\begin{lem}\label{lem:vanishL}
For a $d$-dimensional variety $X$, one has 
$L^n_i(X) = 0 = {{\pi}_i{\wt{C}}_j{L^n}}(X)$ for all $j \ge 0$ and
$i < -d-2$.
\end{lem}
\begin{proof} Using Corollaries~\ref{cor:descentL} and ~\ref{cor:descentL*}
and \cite[Theorem~12.5]{SuslinV}, it suffices to show that
$a_{cdh}{\pi}_q(L^n)$ and $a_{cdh}{\pi}_q\left({\wt{C}}_j{L^n}\right)$
are zero for $i < -2$. The presheaf of fibration sequences
~\ref{eqn:trace} gives the long exact sequence of presheaves of
homotopy groups on $Sch/k$
\[
\cdots \to L^n_i \to {K/{p^n}}_i \to {TC/{p^n}}_i \left(-;p\right)
\to L^n_{i-1} \to \cdots.
\]
Since the sheafification is an exact functor, we get the corresponding
long exact sequence of $cdh$-sheaves  
\begin{equation}\label{eqn:exact1}
\cdots \to a_{cdh}L^n_i \to a_{cdh}{K/{p^n}}_i \to a_{cdh}
{TC/{p^n}}_i \left(-;p\right) \to a_{cdh}L^n_{i-1} \to \cdots.
\end{equation}
We similarly get a long exact sequence of $cdh$-sheaves 
\begin{equation}\label{eqn:exact2}
\cdots \to a_{cdh}{\pi}_i\left({\wt{C}}_j{L^n}\right) \to 
a_{cdh}{\pi}_i\left({\wt{C}}_j{K/{p^n}}\right) \to a_{cdh}
{\pi}_i\left({\wt{C}}_j{TC/{p^n}}\left(-;p\right)\right)
\end{equation}
\[
\hspace*{8cm} \to 
a_{cdh}{\pi}_{i-1}\left({\wt{C}}_j{L^n}\right)  \to \cdots.
\]
Since the smooth schemes have no non-zero negative $K$-theory,
we have $a_{cdh}{K/{p^n}}_i = 0$ for $i < 0$ and hence there are 
isomorphisms
\begin{equation}\label{eqn:exact3}
_{cdh}
{TC/{p^n}}_i \left(-;p\right) \xrightarrow{\cong} a_{cdh}L^n_{i-1} 
\ \ {\rm and}
\end{equation}
\[
{\pi}_i\left({\wt{C}}_j{TC/{p^n}}\left(-;p\right)\right)
\xrightarrow{\cong} a_{cdh}{\pi}_{i-1}\left({\wt{C}}_j{L^n}\right)
\ \ {\rm for} \ \ i < 0.
\]
Thus it suffices to show that the left terms of both the isomorphisms
vanish for $i < -1$. For this, it suffices to show that 
$TC_i\left(A;p, {\Z}/{p^n}\right) = 0$ for $i < -1$ for any ring $A$
which is essentially of finite type over $k$. One knows from
a result of Hesselholt ({\sl cf.} \cite{Hessel1}) that 
$TC_i\left(A;p\right) = 0$ for $i < -1$ and the same conclusion then holds with
finite coefficients by the exact sequence
\[
TC_i\left(A;p\right) \to TC_i\left(A;p, {\Z}/{p^n}\right) \to
TC_{i-1}\left(A;p\right).
\]
\end{proof}  
\begin{lem}\label{lem:spectral1}
Let $X$ be a $k$-variety of dimension $d$. Then there are natural isomorphisms
\[
H^d_{\sC}\left(X, a_{\sC}{\pi}_{-1} 
\left(TC/{p^n}\left(-;p\right)\right)\right) 
\xrightarrow{\cong} {\H}^{d+1}_{\sC}
\left(X, TC/{p^n}\left(-;p\right)\right),
\]
\[
H^d_{\sC}\left(X, a_{\sC}{\pi}_{-1} 
\left({\wt{C}}_j TC/{p^n}\left(-;p\right)\right)\right) 
\xrightarrow{\cong} {\H}^{d+1}_{\sC} \left(X, {\wt{C}}_j 
{TC/{p^n}}\left(-;p\right)\right)
\]
for ${\sC}$ being the Zariski or the $cdh$-site.
\end{lem}
\begin{proof} Since the Zariski or the $cdh$-cohomological dimension
of $X$ is bounded by $d$, one has 
a strongly convergent spectral sequence
\begin{equation}\label{eqn:spec2}
E^{s,t}_2 = H^s_{\sC}\left(X, a_{Zar}{\pi}_t \left(
TC/{p^n}\left(-;p\right)\right)\right)
\Rightarrow {\H}^{s-t}_{\sC}
\left(X, TC/{p^n}\left(-;p\right)\right)
\end{equation}
and the similar spectral sequence holds for the homotopy groups of the
${\wt{C}}_j$ functors. 

Since $TC_i\left(A;p, {\Z}/{p^n}\right) = 0$ for any ring $A$ and for
any $i < -1$ as mentioned above, we conclude that for $s < d$, one has
$-d-1+s < -1$ and hence
\[
H^s_{\sC}\left(X, a_{\sC}{\pi}_{-d-1+s} \left(
TC/{p^n}\left(-;p\right)\right)\right) = 0 =
H^s_{\sC}\left(X, a_{\sC}{\pi}_{-d-1+s} 
\left({\wt{C}}_j TC/{p^n}\left(-;p\right)\right)\right) 
\]
and also $E^{d-2, -2}_2 = 0$. Hence the above spectral sequence degenerates
enough to give the desired isomorphisms.
\end{proof} 
\begin{lem}\label{lem:specL}
Let $X$ be as in Lemma~\ref{lem:spectral1}. Then the natural maps
\[
H^d_{cdh}\left(X, a_{cdh}{\pi}_{-2}(L^n)\right)
\to {\H}^{d+2}_{cdh}
\left(X, L^n\right),
\]
\[
H^d_{cdh}\left(X, a_{cdh}{\pi}_{-2} 
\left({\wt{C}}_j L^n\right)\right) 
\to {\H}^{d+2}_{cdh} \left(X, {\wt{C}}_j L^n\right)
\]
are isomorphisms.
\end{lem}
\begin{proof} Since the smooth schemes have no negative $K$-theory,
we have $a_{cdh}{K/{p^n}}_i = 0$ for $i < 0$. We have also
seen above that $a_{cdh}{TC/{p^n}}_i \left(-;p\right) = 0$
for $i < -1$, and the same vanishing holds for the homotopy groups of the
${\wt{C}}_j$-functors. We conclude from the exact sequences 
~\ref{eqn:exact1} and ~\ref{eqn:exact2} that
$a_{cdh}L^n_{i} = 0 = a_{cdh}{\pi}_{i}\left({\wt{C}}_j{L^n}\right)$
for $i < -2$. The spectral sequence ~\ref{eqn:spec2} now implies the
lemma.
\end{proof} 

\section{Vanishing and homotopy invariance for $K/{p^n}$}
In this section, we prove the vanishing results for some negative
homotopy groups of $K/{p^n}$ and ${\wt{C}}_j {K/{p^n}}$ using the
results of the previous section. We begin with the following result
about the sheaves of rings of Witt vectors. For a $k$-variety $X$,
let $W{\sO}_X$ denote the sheaf of Witt vectors on $X$ ({\sl cf.}
\cite{Illusie}). The sheaf $W{\sO}$ is a sheaf of rings on the category
$Sch/k$.
\begin{lem}\label{lem:Wvectors}
For any $k$-variety of dimension $d$ and for any $n \ge 1$, the natural maps
\begin{equation}\label{eqn:Witt1}
H^d_{Zar}\left(X, {W{\sO}_X}/{p^n}\right) \to
H^d_{cdh}\left(X, {W{\sO}_X}/{p^n}\right) 
\end{equation}
\begin{equation}\label{eqn:Witt2}
H^d_{Zar}\left(X, {\wt{C}}_j {W{\sO}_X}/{p^n}\right) \to
H^d_{cdh}\left(X, {\wt{C}}_j {W{\sO}_X}/{p^n}\right)
\end{equation}
are surjective.
\end{lem}
\begin{proof} We prove by induction on $n \ge 1$.
For $n = 1$, there is a canonical surjective map of ${\F}_p$-algebras
${W{\sO}}/{p} \surj {\sO}$ on $Sch/k$ such that the kernel ${\sI}$ is a sheaf 
of square zero ideals by \cite[page 9]{Lubkin}. Hence by
\cite[Lemma~12.1]{SuslinV}, the map
$H^i_{cdh}\left(X, {W{\sO}_X}/{p}\right) \to 
H^i_{cdh}\left(X, {\sO}_X\right)$ is an isomorphism for all $i \ge 0$. 
In particular, we get $H^d_{cdh}\left(X, {\sI}_X\right) = 0$. 
Thus we get a 
commutative diagram of cohomology groups with exact rows
\[
\xymatrix{
H^d_{Zar}\left(X, {\sI}_X\right) \ar[r] & 
H^d_{Zar}\left(X, {W{\sO}_X}/{p}\right) \ar[r] \ar[d] &
H^d_{Zar}\left(X, {\sO}_X\right) \ar[r] \ar[d] & 0 \\
& H^d_{cdh}\left(X, {W{\sO}_X}/{p}\right) \ar[r]^{\cong} &
H^d_{cdh}\left(X, {\sO}_X\right) & }
\]
The right vertical map is surjective by \cite[Theorem~6.1]{CHSW}.
We remark here that Theorem~6.1 of \cite{CHSW} is proved when the
base field is of characteristic zero. However exactly the same 
argument works even if $k$ is of positive characteristic as long as
the resolution of singularities holds over $k$, which we have assumed
throughout. Only extra ingredients needed are the results of
\cite{Thomason3} and the formal function theorem which are characteristic
free. We refer to {\sl loc. cit.} for details of the proof.  
We conclude from the above diagram that ~\ref{eqn:Witt1} holds for $n = 1$.

Now we assume that $n \ge 2$ and the surjectivity holds in 
~\ref{eqn:Witt1} for $m \le n-1$.
We have the exact sequence of the sheaves of abelian groups on $Sch/k$
\[
0 \to {W{\sO}}/{p^{n-1}} \to {W{\sO}}/{p^n} \to {W{\sO}}/{p} \to 0.
\]
Considering the cohomology, we get the following commutative diagram 
with exact rows (since $H^d$ is right exact on the Zariski or the
$cdh$-site of $X$).
\[
\xymatrix{
H^d_{Zar}\left(X, {W{\sO}_X}/{p^{n-1}}\right) \ar[r] \ar[d] &
H^d_{Zar}\left(X, {W{\sO}_X}/{p^n}\right) \ar[r] \ar[d] &
H^d_{Zar}\left(X, {W{\sO}_X}/{p}\right) \ar[r] \ar[d] & 0 \\ 
H^d_{cdh}\left(X, {W{\sO}_X}/{p^{n-1}}\right) \ar[r] &
H^d_{cdh}\left(X, {W{\sO}_X}/{p^n}\right) \ar[r] &
H^d_{cdh}\left(X, {W{\sO}_X}/{p}\right) \ar[r] & 0}
\]
The right and the left vertical arrows are surjective by induction.
Hence the middle vertical arrow is also surjective. This proves
the surjectivity of the map in ~\ref{eqn:Witt1}.

To prove the surjectivity of the map in ~\ref{eqn:Witt2}, we apply
exactly the same argument of induction as above and observe that
for $n = 1$, the map
${\wt{C}}_j W{\sO} \to {\wt{C}}_j {\sO}$ is surjective such that its kernel
${\wt{C}}_j {\sI}$ is of the form ${\sI} {\otimes}_{{\F}_p} V$ for some 
${\F}_p$-vector space $V$, and hence $H^d_{cdh}\left(X, {\wt{C}}_j {\sI}
\right) \cong H^d_{cdh}\left(X, {\sI}\right) {\otimes}_{{\F}_p} V = 0$.
In particular, we get isomorphism 
$H^d_{cdh}\left(X, {\wt{C}}_j {W{\sO}_X}/{p}\right) 
\xrightarrow{\cong} H^d_{cdh}\left(X, {\wt{C}}_j {\sO}_X\right)$.
Moreover, the map $H^d_{Zar}\left(X, {\wt{C}}_j {\sO}_X\right)
\to H^d_{cdh}\left(X, {\wt{C}}_j {\sO}_X\right)$ is surjective
by \cite[Proof of Theorem~6.2]{CHSW}. The proof now follows from the
induction.
\end{proof}
\begin{prop}\label{prop:crucial}
Let $X$ be a $k$-variety of dimension $d$. Then the natural maps
\[
H^d_{Zar}\left(X, a_{Zar}{\pi}_{-1} \left(
TC/{p^n}\left(-;p\right)\right)\right) \to
H^d_{cdh}\left(X, a_{cdh}{\pi}_{-1} \left(
TC/{p^n}\left(-;p\right)\right)\right);
\]
\[
H^d_{Zar}\left(X, a_{Zar}{\pi}_{-1} 
\left({\wt{C}}_j TC/{p^n}\left(-;p\right)\right)\right) \to 
H^d_{cdh}\left(X, a_{cdh}{\pi}_{-1} 
\left({\wt{C}}_j TC/{p^n}\left(-;p\right)\right)\right)
\]
are surjective.
\end{prop}  
\begin{proof} Since $a_{\sC} {\pi}_{-2} \left(TC\left(-;p\right)\right)
= 0$ for $\sC$ being the Zariski or the $cdh$-cite, the exact
sequence
\[
a_{\sC} {\pi}_{-1} \left(TC\left(-;p\right)\right) \xrightarrow{p^n}
a_{\sC} {\pi}_{-1} \left(TC\left(-;p\right)\right) \to
a_{\sC}{\pi}_{-1} \left(TC/{p^n}\left(-;p\right)\right) \to
a_{\sC} {\pi}_{-2} \left(TC\left(-;p\right)\right)
\]
implies that $a_{\sC} {\pi}_{-1} \left({TC/{p^n}}\left(-;p\right)\right)
\cong {a_{\sC} {\pi}_{-1} \left(TC\left(-;p\right)\right)}/{p^n}$
and the same isomorphism holds for the ${\wt{C}}_j$-sheaves.
On the other hand, it is known ({\sl cf.} \cite{Hessel1}) that 
there are canonical isomorphisms
\[
a_{\sC}{\pi}_{-1} \left(TC\left(-;p\right)\right) \xrightarrow{\cong}
{\rm Coker}\left(a_{\sC}{W{\sO}_X} \xrightarrow{id-F}
a_{\sC}{W{\sO}_X}\right)
\] 
and similar isomorphism holds for the ${\wt{C}}_j$-sheaves, where
\[
{\wt{C}}_j W{\sO}_X = {\rm CoKer}\left(W{\sO}_{X[T_1, \cdots , T_j]}
\to W{\sO}_X\right).
\]
Combining the above two isomorphisms, we conclude that there
are canonical isomorphisms
\begin{equation}\label{eqn:Witt}
a_{Zar}{\pi}_{-1} \left(TC/{p^n}\left(-;p\right)\right)
\xrightarrow{\cong}
{\rm Coker}\left(a_{Zar}{W{\sO}_X}/{p^n} \xrightarrow{id-F}
a_{Zar}{W{\sO}_X}/{p^n}\right) \ \ {\rm and}
\end{equation}
\[
a_{cdh}{\pi}_{-1} \left(TC/{p^n}\left(-;p\right)\right)
\xrightarrow{\cong}
{\rm Coker}\left(a_{cdh}{W{\sO}_X}/{p^n} \xrightarrow{id-F}
a_{cdh}{W{\sO}_X}/{p^n}\right).
\]
The similar isomorphisms hold for the ${\wt{C}}_j$-sheaves. 
Since $H^d$ is right exact on the Zariski or the $cdh$-site of $X$,
it suffices to show that for every $n \ge 1$, the maps
\begin{equation}\label{eqn:Witt1*}
H^d_{Zar}\left(X, {W{\sO}_X}/{p^n}\right) \to
H^d_{cdh}\left(X, {W{\sO}_X}/{p^n}\right) 
\end{equation}
\begin{equation}\label{eqn:Witt2*}
H^d_{Zar}\left(X, {\wt{C}}_j {W{\sO}_X}/{p^n}\right) \to
H^d_{cdh}\left(X, {\wt{C}}_j {W{\sO}_X}/{p^n}\right)
\end{equation}
are surjective. But this is proved in Lemma~\ref{lem:Wvectors}.
\end{proof}
\begin{thm}\label{thm:main*}
Let $X$ be a $k$-variety of dimension $d$. Then 
$K_i\left(X, {\Z}/{p^n}\right) = 0 = 
{{\pi}_i {\wt{C}}_j{K/{p^n}}}(X)$ for $j \ge 0$ and $i < -d-1$.
\end{thm}
\begin{proof} We first show that 
\begin{equation}\label{eqn:vanishTC}
TC_i\left(X;p, {\Z}/{p^n}\right) = 0 =
{{\pi}_i {\wt{C}}_j 
{TC/{p^n}}\left(-;p\right)}(X) \ \ {\rm for} \ \ 
i < -d-1.
\end{equation}   
The Zariski descent for $TC/{p^n}\left(-;p\right)$ and 
${\wt{C}}_j {TC/{p^n}} \left(-;p\right)$,
as shown in the proof of Corollary~\ref{cor:descentL},
gives a strongly convergent spectral sequence
\[
E^{s,t}_2 = H^s_{Zar}\left(X, a_{Zar}{\pi}_t \left(
TC/{p^n}\left(-;p\right)\right)\right)
\Rightarrow TC_{t-s}\left(X;p, {\Z}/{p^n}\right)
\]
and the similar spectral sequence holds for the homotopy groups of the
${\wt{C}}_j$ functors.
Thus it suffices to show that  
$TC_i\left(A;p, {\Z}/{p^n}\right) = 0$ for $i < -1$ for any ring $A$
which is essentially of finite type over $k$. But this has already been
shown above. This proves ~\ref{eqn:vanishTC}. 

The homotopy fibration sequence ~\ref{eqn:trace}
gives the long exact sequence of homotopy groups
\[
\cdots L^n_i(X) \to K_i\left(X, {\Z}/{p^n}\right) \to
TC_i\left(X;p, {\Z}/{p^n}\right) \to L^n_{i-1}(X) \to \cdots
\]
and one has a similar long exact sequence of the homotopy groups
of the functors ${\wt{C}}_j$.
Lemma~\ref{lem:vanishL} and ~\ref{eqn:vanishTC} together now imply that 
\[
K_i\left(X, {\Z}/{p^n}\right) = 0 =
{{\pi}_i{\wt{C}}_j{K/{p^n}}}(X) \ \ {\rm for} \ \ i < -d-2
\]
and there are exact sequences
\[
TC_{-d-1}\left(X;p, {\Z}/{p^n}\right) \to
L^n_{-d-2}(X) \to K_{-d-2}\left(X, {\Z}/{p^n}\right) \to 0,
\]
\[
{{\pi}_{-d-1} {\wt{C}}_j {TC/{p^n}}\left(-;p\right)}(X) \to
{{\pi}_{-d-2} {\wt{C}}_j L^n(X) \to
{{\pi}_{-d-2} {\wt{C}}_j{K/{p^n}}}}(X) \to 0.
\]
Thus we need to show that the first map in both the exact sequences
are surjective.  

We consider the following commutative diagram.
\[
\xymatrix@C.6pc{
H^d_{Zar}\left(X, a_{Zar}{\pi}_{-1} \left(
TC/{p^n}\left(-;p\right)\right)\right) \ar[d] \ar[r] &
{\H}^{d+1}_{Zar}
\left(X, TC/{p^n}\left(-;p\right)\right) \ar[d] &
TC_{-d-1}\left(X; p, {\Z}/{p^n}\right) \ar[l] \ar[dd] \\
H^d_{cdh}\left(X, a_{cdh}{\pi}_{-1} \left(
TC/{p^n}\left(-;p\right)\right)\right) \ar[d] \ar[r] &
{\H}^{d+1}_{cdh} \left(X, TC/{p^n}\left(-;p\right)\right) \ar[d] & \\
H^d_{cdh}\left(X, a_{cdh} {\pi}_{-2}(L^n)\right) \ar[r] &
{\H}^{d+2}_{cdh}\left(X, L^n\right) & L^n_{-d-2}(X) \ar[l] }
\]
The left horizontal arrows of all the rows are isomorphisms by 
Lemmas ~\ref{lem:spectral1} and ~\ref{lem:specL}. The right horizontal
arrow of the top row is an isomorphism by the Zariski descent of 
$TC$ ({\sl cf.} \cite{GH1}, \cite{BM}). The right horizontal arrow in the
bottom row is an isomorphism by Corollary~\ref{cor:descentL}.
The lower vertical arrow on the left column is an isomorphism by
~\ref{eqn:exact3}. The upper vertical arrow on the left column is 
surjective by Proposition~\ref{prop:crucial}. A diagram chase shows that
the long vertical arrow in the extreme right is surjective.
The surjectivity  of the map ${{\pi}_{-d-1} {\wt{C}}_j 
{TC/{p^n}}\left(-;p\right)}(X) \to {{\pi}_{-d-2} {\wt{C}}_j L^n(X)}$
follows exactly in the same way using Lemmas ~\ref{lem:spectral1}, 
~\ref{lem:specL}, Proposition~\ref{prop:crucial} and 
Corollary~\ref{cor:descentL}.
\end{proof}

\section{Vanishing and homotopy invariance for rational $K$-theory}
For any presheaf of spectra $\sE$ on $Sch/k$, let ${\sE}_{\Q}$ denote
the direct colimit over the multiplication maps
${\sE} \xrightarrow{n} {\sE}$ by positive integers ({\sl cf} 
\cite{Thomason2}). Then ${\sE}_{\Q}$ is a presheaf of spectra on
$Sch/k$ such that ${\pi}_i\left({\sE}_{\Q}\right) \cong 
{\pi}_i\left(\sE\right) {\otimes}_{\Z} {\Q}$ for $i \in {\Z}$.
Our goal now is prove the vanishing of the rational $K$-theory and
$K_{\Q}$-regularity in degrees below minus the dimension of a
$k$-variety, where $k$ is an infinite perfect field of positive
characteristic as before. Let ${\sK}_{\Q}$ and ${\sH}{\sC}_{\Q}$
denote the presheaves of rational $K$-theory and rational cyclic
homology spectra on $Sch/k$. Let ${\wt{{\sH}{\sN}}}_{\Q}$,
${\wt{{\sH}{\sP}}}_{\Q}$ and ${\wt{{\sH}{\sC}}}_{\Q}$ denote the presheaves
of spectra on $Sch/k$ given by $U \mapsto HN\left(U {\otimes} {\Q}\right)$,
$U \mapsto HP\left(U {\otimes} {\Q}\right)$ and
$U \mapsto HC\left(U {\otimes} {\Q}\right)$, where $U$ is considered
as a scheme over $\Z$, and $HN$, $HP$ and $HC$ respectively are the presheaves
of negative cyclic homology, periodic cyclic homology and cyclic homology
spectra on the category of schemes over $\Z$.
There is a homotopy fibration sequence of the Eilenberg-Mac Lane spectra
({\sl cf.} \cite[Section~5.1]{Loday})
\begin{equation}\label{eqn:rational}
{\wt{{\sH}{\sN}}}_{\Q} \to {\wt{{\sH}{\sP}}}_{\Q} \to
{\Omega}^{-2} {\wt{{\sH}{\sC}}}_{\Q}.
\end{equation}
There is a generalized Chern character map ({\sl cf.} 
\cite[Section~8.4]{Loday})
\[
{\sK}_{\Q} \xrightarrow{ch} {\wt{{\sH}{\sN}}}_{\Q}.
\]
Let ${\sK}^{\inf}_{\Q}$ denote the homotopy fiber of the above map
of spectra. 
\begin{lem}\label{lem:Rdescent}
The presheaf of spectra ${\sK}_{\Q}$ on $Sch/k$ satisfies $cdh$-descent.
\end{lem}
\begin{proof} Since our schemes are defined over $k$ which is of positive
characteristic, we see that the presheaf of spectra ${\wt{{\sH}{\sN}}}_{\Q}$
is contractible. Hence using the homotopy fibration sequence
\[
{\sK}^{\inf}_{\Q} \to {\sK}_{\Q} \xrightarrow{ch} {\wt{{\sH}{\sN}}}_{\Q},
\]
it suffices to show that the presheaf of spectra ${\sK}^{\inf}_{\Q}$
satisfies $cdh$-descent on $Sch/k$. To this end, it suffices to show that
${\sK}^{\inf}_{\Q}$ satisfies all the conditions of Theorem~\ref{thm:descent}.
It satisfies excision by \cite[Theorem~01]{Cortinas} and it is invariant
under infinitesimal extension by \cite[Main Theorem]{Goodwillie}.
${\sK}^{\inf}_{\Q}$ satisfies Nisnevich descent by 
\cite[Theorem~10.8]{Thomason2} and it satisfies the Mayer-Vietoris
property for blow-up under regular closed embeddings by 
\cite[Theorem~2.1]{Thomason3}. This completes the proof of the lemma.
\end{proof}
\begin{cor}\label{cor:RK} Let $X$ be a $k$-variety of dimension $d$.
Then one has $K_i(X) {\otimes}_{\Z} {\Q} = 0 =
{\pi}_i\left({\wt{C}}_j {\sK}_{\Q}\right)(X)$ for $j \ge 0$ and $i < -d$.
\end{cor}
\begin{proof} This follows directly from Lemma~\ref{lem:Rdescent},
the spectral sequence of Corollary~\ref{cor:descentL*} and from
\cite[Theorem~12.5]{SuslinV} since $a_{cdh}{\pi}_i ({\sK}_{\Q}) = 0$
for $i < 0$.
\end{proof}   
\begin{lem}\label{lem:KH} 
Let ${\sK}{\sH}$ denote the presheaf of 
homotopy invariant $K$-theory on $Sch/k$ ({\sl cf.} \cite{Weibel3}).
Then ${\sK}{\sH}$ satisfies $cdh$-descent on $Sch/k$.
\end{lem}
\begin{proof} The proof of the lemma follows exactly in the same way
as the case when when $k$ is of characteristic zero 
({\sl cf.} \cite[Theorem~6.4]{Haes}). Only thing one needs is that the
resolution of singularities should hold over $k$. We skip the details.
\end{proof}
\begin{cor}\label{cor:finite}
Let $X$ be a $k$-variety of dimension $d$ and let $n$ be a positive
integer prime to $p$. Then $K_i\left(X, {\Z}/n\right) = 0 =
{\wt{C}}_jK_i\left(X, {\Z}/n\right)$ for $j \ge 0$ and $i < -d$.
\end{cor}   
\begin{proof} The statement of the corollary for $KH_i(X)$
and ${\wt{C}}_jKH_i(X)$ follows from Lemma~\ref{lem:KH} and
Corollary~\ref{cor:descentL*}. For $n$ prime to $p$, the natural
map $K_i\left(X, {\Z}/n\right) \to KH_i\left(X, {\Z}/n\right)$ is
an isomorphism by \cite[Proposition~1.6]{Weibel3}, 
and the same conclusion holds for ${\wt{C}}_j K_i(X)$.
Now the corollary follows from the exact sequence
\[
0 \to KH_i(X) {\otimes} {\Z}/n \to
KH_i\left(X, {\Z}/n\right) \to {\rm Tor}\left(KH_{i-1}(X), {\Z}/n\right)
\to 0.\]
\end{proof} 
{\bf{Proof of Theorem~\ref{thm:main}:}}
For any abelian group $A$, let $_nA$ denote the subgroup of $n$-torsion
elements of $A$.
It follows from Corollary~\ref{cor:RK} that $K_i(X)$
and ${\wt{C}}_jK_i(X)$ are torsion groups for $i < -d$. Thus we only need to 
show that these groups have no torsion whenever $i < -d-2$.
For any positive integer $n$, there is a short exact sequence
\[ 
0 \to K_i(X) {\otimes} {\Z}/n \to
K_i\left(X, {\Z}/n\right) \to _nK_{i-1}(X) \to 0.
\]
Theorem~\ref{thm:main*} and Corollary~\ref{cor:finite}
now immediately imply that $K_i(X)$ and ${\wt{C}}_jK_i(X)$ are divisible
groups for $i < -d-1$. Moreover, these groups are torsion-free
for $i < -d-2$. In particular, $K_i(X) = 0$ and $X$ is $K_i$-regular for 
$i < -d-2$.
$\hfill\square$
\\
\\ 

School of Mathematics \\
Tata Institute Of Fundamental Research \\
Homi Bhabha Road \\
Mumbai,400005, India \\
email : amal@math.tifr.res.in \\
\end{document}